\documentclass[11pt]{amsart}

\usepackage{amsfonts}
\usepackage{amsmath}
\usepackage{amssymb}
\usepackage{amscd}
\usepackage{color}
\usepackage{xcolor}
\usepackage{amsthm}
\usepackage[hmargin=4cm,vmargin=4cm]{geometry}
\usepackage{parskip}

\usepackage{enumitem}

\newtheorem{thm}{Theorem}

\newtheorem{thm*}{Theorem}
\newtheorem{prop}{Proposition}

\theoremstyle{definition}

\theoremstyle{remark}

\newtheorem{rmk}[prop]{Remark} 

\newcommand{\F}{{\mathbb{F}}}
\newcommand{\R}{{\mathbb{R}}}
\newcommand{\Z}{{\mathbb{Z}}}

\newcommand{\bK}{{\mathbb{K}}}
\newcommand{\bF}{{\mathbb{F}}}

\newcommand{\sm}[1]{C^\infty(#1)}

\newcommand{\til}[1]{\widetilde{#1}}

\newcommand{\ol}[1]{\overline{#1}}

\newcommand{\zt}{{\Z/(2)}}

\newcommand{\om}{\omega}

\newcommand{\cA}{\mathcal{A}}

\newcommand{\cH}{\mathcal{H}}

\newcommand{\fix}{\mathrm{Fix}}

\newcommand{\pr}{pseudo-rotation}

\def\mrm#1{{\mathrm{#1}}}

\def\cl#1{{\mathcal{#1}}}

\DeclareMathOperator{\Ham}{\mathrm{Ham}}

\DeclareMathOperator{\loc}{\mathrm{loc}}

\def\H2{H^{(2)}}

\begin{document}

\title{Pseudo-rotations and Steenrod squares revisited}

\author{Egor Shelukhin}
\date{}

\begin{abstract}
In this note we prove that if a closed monotone symplectic manifold $M$ admits a Hamiltonian \pr, which may be degenerate, then the quantum Steenrod square of the cohomology class Poincar\'{e} dual to the point must be deformed. This result gives restrictions on the existence of {\pr}s, implying a form of uniruledness by pseudo-holomorphic spheres, and generalizes a recent result of the author. The new component in the proof consists in an elementary calculation with capped periodic orbits.

\end{abstract}


\maketitle

\section{Setup}





In this paper, $(M,\om)$ denotes a closed monotone symplectic manifold  of dimension $2n,$ with the symplectic form rescaled so that $[\om] = 2 c_1(TM)$ on the image of the Hurewicz map $\pi_2(M) \to H_2(M;\Z).$ For a Hamiltonian diffeomorphism $\phi \in \Ham(M,\om),$ we denote by $\fix(\phi)$ the set of its {\em contractible} fixed points, and by $x^{(k)}$ for $x \in \fix(\phi)$ its image under the inclusion $\fix(\phi) \subset \fix(\phi^k).$ Contractible means that the homotopy class $\alpha(x,\phi)$ of the path $\alpha(x,H) = \{ \phi^t_H(x)\}$ for a Hamiltonian $H \in \cl{H} \subset \sm{[0,1] \times M, \R}$ generating $\phi$ as the time-one map $\phi^1_H = \phi$ of its Hamiltonian flow, is trivial\footnote{This class does not depend on the choice of Hamiltonian by a classical argument in Floer theory.}. Here $\cl{H} = \cap_{t \in [0,1]} \ker(I_t),$ $I_t(H) = \int_M H(t,-)\, \om^n.$ 

To an isolated fixed point $x$ one associates (cf. \cite{GG-negmon}) a local cohomology group $HF_{\loc}(\phi,x),$ which is naturally $\zt$-graded. If we choose a capping $\ol{x}$ of $\alpha(x,H),$ we obtain a $\Z$-graded version, $HF_{\loc}^{\ast}(H,\ol{x}) = HF_{\loc}^{\ast}(\til{\phi},\ol{x}).$ It depends only on the class $\til{\phi}$ of $\{\phi^t_H\}_{t \in [0,1]}$ in the universal cover $\til{\Ham}(M,\om),$ and the capped orbit $\ol{x}.$

We say that $\phi \in \Ham(M,\om)$ is an $\F_2$ {\em Hamiltonian \pr} if:
\begin{enumerate}[label = (\roman*)]
\item \label{cond: perfect} It is perfect, that is for all iterations $k \geq 1$ of $\phi,$ $\fix(\phi^k) = \fix(\phi).$ In other words, $\phi$ admits no simple periodic orbits of order $k >1.$

\item \label{cond: sum of Betti} for all iterations $k\geq 1,$ \[\displaystyle N(\phi^k,\F_2) = \sum_{x \in \fix(\phi)} \dim_{\F_2} HF^*_{\loc}(\phi^k,x^{(k)}) = \dim_{\F_2} H^*(M; \F_2).\] 
\end{enumerate}

\begin{rmk}\label{rmk: pseudorot}
Observe that a perfect Hamiltonian diffeomorphism necessarily has no symplectically degenerate maxima (see \cite{GG-revisited}). Furthermore, if $\phi$ is strongly non-degenerate, that is all the points in $\fix(\phi^k)$ are non-degenerate, for all $k \geq 1,$ then $HF_{\loc}(\phi,x) \cong \bK$, and all iterations are {\em admissible}, that is $\lambda^k \neq 1$ for all eigenvalues $\lambda \neq 1$ of $D(\phi)_x$. By the Smith inequality in local Floer homology \cite{CineliGinzburg, SZhao-pants}, conditions \ref{cond: perfect} and \ref{cond: sum of Betti} imply for iterations $k = 2^m,$ the stronger statement that for all $x \in \fix(\phi),$ $\dim_{\F_2} HF^{\loc}(\phi^k,x^{(k)}) = \dim_{\F_2} HF^{\loc}(\phi, x).$ Moreover, \cite[Theorem A]{S-HZ} indicates that when a Hamiltonian diffeomorphism has a finite number of periodic points, then a condition like \ref{cond: sum of Betti} should be satisfied. Showing this in general would bridge the gap between the initial Chance-McDuff conjecture (see for example \cite{GG-revisited}) and the main result of this note, Theorem \ref{thm: uniruled}. Dynamics of Hamiltonian {\pr}s in higher dimensions were recently studied by Ginzburg and G\"{u}rel \cite{GG-pseudorotations}. We refer thereto for further discussion of this interesting notion, and survey results more closely related to the subject of this paper in Section \ref{sec: results}.
\end{rmk}


We remind the reader that the minimal Chern number of $(M,\om)$ is the index \[N = N_M = [\Z:\mrm{image}(c_1(TM): \pi_2(M) \to \Z)].\] The notion of mean-index, introduced in symplectic topology in \cite{SalamonZehnder} is described as follows. For a Hamiltonian $H \in \cl{H}$ generating $\til{\phi} \in \til{\Ham}(M,\om)$ and capped periodic orbit $\ol{x}$ of $H,$ we set \[\Delta(H,\ol{x}) = \Delta(\til{\phi}_H,\ol{x}) = \lim_{k \to \infty} \frac{1}{k} CZ(\til{\phi}^k, \ol{x}^{(k)}),\] where $\ol{x}^{(k)}$ is $\ol{x}$ iterated $k$ times, which is indeed a capped periodic orbit of a Hamiltonian generating $\til{\phi}^k.$ The limit exists, since the Conley-Zehnder index comes from a quasi-morphism $\til{Sp}(2n,\R) \to \R$ (see \cite{EntovPolterovich-rigid}). The mean-index satisfies the following properties that we use below:

\begin{enumerate}
	\item {\em homogeneity:} $\Delta(\til{\phi}^k,\ol{x}^{(k)}) = k \cdot \Delta(\til{\phi},\ol{x}),$ for all $k \in \Z_{>0}.$
	\item {\em recapping:} $\Delta(\til{\phi},\ol{x} \# A) = \Delta(\til{\phi},\ol{x}) - 2\left<c_1(TM),A\right>,$ $A \in H_2^S(M;\Z).$
	\item {\em distance to index:} $CZ(\til{\phi},\ol{x}) \in [\Delta(\til{\phi},\ol{x}) - n, \Delta(\til{\phi},\ol{x}) + n],$ 
	
	
	\item {\em support of local Floer cohomology:} $HF_{\loc}^r(\til{\phi},\ol{x}) = 0,$ unless $r \in [\Delta(\til{\phi},\ol{x}) - n, \Delta(\til{\phi},\ol{x}) + n].$
	
	\item {\em symplectically degenerate maxima:} by definition, $\ol{x}$ is not a SDM of $\phi,$ if $HF_{\loc}^r(\til{\phi},\ol{x}) = 0,$ unless $r \in [\Delta(\til{\phi},\ol{x}) - n, \Delta(\til{\phi},\ol{x}) + n).$ 
\end{enumerate}

\begin{rmk}
	Observe that, by the symmetry of the Conley-Zehnder index and a simple duality argument, if the reversal $\ol{x}^{(-1)}$ of $\ol{x}$ is not an SDM of $\phi^{-1},$ then $HF_{\loc}^r(\til{\phi},\ol{x}) = 0,$ unless $r \in (\Delta(\til{\phi},\ol{x}) - n, \Delta(\til{\phi},\ol{x}) + n].$
\end{rmk}


Finally, for a quantum cohomology class $\mu \in QH^*(M;\Lambda_{\F_2}) \setminus \{0\}$ and a Hamiltonian $H \in \cH$ with isolated contractible fixed points $\fix(\phi^1_H),$ we recall that the Hamiltonian spectral invariant $c(\mu,H)$ of $\mu$ is {\em carried} by a capped $1$-periodic orbit $\ol{x}$ of $H,$ if in a suitable sense $\ol{x}$ is a lowest action term in a highest minimal action representative of the image $PSS_{H}(\mu)$ of $\mu$ under the PSS isomorphism \cite{PSS} from the quantum cohomology $QH^*(M,\Lambda_{\F_2})$ to the filtered Floer cohomology of the Hamiltonian $H.$ Keeping in mind the duality between Floer homology and Floer cohomology \cite{LeclercqZapolsky}, we refer the reader to \cite{GG-negmon, GG-revisited} for a detailed description of this notion, recording only the following two facts:

\begin{enumerate}
	\item {\em spectrality:} for $(M,\om)$ rational, in particular monotone, for each non-zero class $\mu \in QH^*(M;\Lambda_{\F_2}),$ and $H \in \cH$ with $\# \fix(\phi^1_H) < \infty,$ $c(\mu,H)$ is carried by at least one capped $1$-periodic orbit $\ol{x}$ of $H.$
	
	\item {\em contribution to local Floer cohomology:} if $\mu$ is homogeneous of degree $k,$ and $\ol{x}$ carries $c(\mu, H),$ then $HF^k_{\loc}(H,\ol{x}) \neq 0.$ 

\end{enumerate}



\section{Results}\label{sec: results}

We call a symplectic manifold strongly uniruled if there exists a non-trivial $3$-point genus-$0$ Gromov-Witten invariant $\left< [pt], a, b\right>_{\beta},$ for $\beta \in H_2(M,\Z) \setminus \{0\}.$ By \cite[Lemma 2.1]{McDuff-uniruled}, if $(M,\om)$ is not strongly uniruled, then the quantum square $\mu \ast \mu = 0$ for the degree $2n$ cohomology class $\mu$ Poincar\'{e} dual to the point class. A generally different stronger notion than $\mu \ast \mu = 0,$ is that the quantum Steenrod square  $\cl{QS}(\mu),$ defined in \cite{Wilkins}, of the volume class $\mu$ satisfy \begin{equation}\label{eq: non St uni} \cl{QS}(\mu) = h^{2n} \mu.\end{equation} Note that $\cl{QS}(\mu)$ is equal to the classical Steenrod square $Sq(\mu) = h^{2n} \mu$ plus quantum corrections coming from certain pseudo-holomorphic curves. When \eqref{eq: non St uni} does not hold, we say that $M$ is {\em $\zt$-Steenrod uniruled}. The main result of this note is the following.

\medskip

\begin{thm}\label{thm: uniruled}
Let $(M,\om)$ be a closed monotone symplectic manifold admitting an $\F_2$ Hamiltonian pseudo-rotation $\phi.$ Then $(M,\om)$ is $\zt$-Steenrod uniruled. 
\end{thm}


%
%



\begin{rmk}\label{rmk: Steenrod basic}
	We observe following \cite{S-PRQS} that when $(M,\om)$ is $\zt$-Steenrod uniruled, then there exists a $J$-holomorphic curve through each point of $M$ for each $\om$-compatible almost complex structure $J.$ Hence, Theorem \ref{thm: uniruled} provides a geometric obstruction to the existence of pseudo-rotations. Other arguments described in \cite{S-PRQS} provide obstructions to the existence of pseudo-rotations for monotone manifolds with $N>n,$ and rule it out completely for $N>n+1.$ We remark that the existence of pseudo-rotations was ruled out in \cite{GG-revisited} (and references therein), in a strong way, by proving the Conley conjecture, for manifolds such that $\om(A) \cdot c_1(A) \leq 0$ for all spherical homology classes $A \in H_2^S(M;\Z).$ 
\end{rmk} 

\begin{rmk}
Theorem \ref{thm: uniruled} was proven by the author in \cite{S-PRQS} under the additional assumption that $(M,\om)$ satisfies the Poincar\'{e} duality property for Hamiltonian spectral invariants (which holds in particular whenever the minimal Chern number of $(M,\om)$ satisfies $N > n$). Under the assumption that $\phi$ is strongly non-degenerate, Theorem \ref{thm: uniruled} was also proved by \c{C}ineli, Ginzburg, and G\"{u}rel in \cite{CGG2}, using different additional arguments extending \cite{S-PRQS}. A related, but quite different, result relating the existence of pseudo-rotations to pseudo-holomorphic spheres was shown in \cite{CGG}. 
\end{rmk}





\medskip

\section{Proof}

The proof of the main result relies on the following observations. First, as mentioned in Remark \ref{rmk: pseudorot}, no fixed point of $\phi$ or $\phi^{-1}$ is a symplectically degenerate maximum (see \cite{GG-revisited}). In particular, if the capping $\ol{x}$ of a contractible fixed point $x \in \fix(\phi)$ carries a cohomology class $\mu$ of Conley-Zehnder index $n$ in $HF^n(\til{\phi}) \cong QH^{2n}(M,\Lambda_{\F_2})$ for a lift $\til{\phi}$ of $\phi$ to the universal cover $\til{\Ham}(M,\om)$ of $\Ham(M,\om),$ then its mean-index $\Delta = \Delta(\til{\phi},\ol{x})$ satisfies $n< \Delta + n.$ Hence \begin{equation}\label{eq: Delta pos} \Delta(\til{\phi},\ol{x}) > 0.\end{equation}

%


%

Second, the following result was proven in \cite{S-PRQS} specifically in the setting of a {\em pseudo-rotation} assuming that \eqref{eq: non St uni} holds. Here $\mu \in QH^{2n}(M,\Lambda_{\F_2})$ denotes the cohomology class Poincar\'{e} dual to the point.

%
\begin{thm}\label{thm: spec ineq}
Let $\psi$ be an $\bF_2$ {\pr} of $(M,\om)$ that is not $\zt$-Steenrod uniruled. Then \begin{equation}\label{eq: greater or equal} c(\mu,\til{\psi}^2) \geq 2 \cdot c(\mu, \til{\psi})\end{equation} for each $\til{\psi} \in \til{\Ham}(M,\om)$ covering $\psi.$
\end{thm}

We proceed to the proof of the main result, which follows a calculation from \cite{GG-negmon}.

\begin{proof}[Proof of Theorem \ref{thm: uniruled}] 

Choose $H \in \cl H,$ such that the path $\{\phi^t_H \}_{t \in [0,1]}$ represents the class $\til{\phi}$ lifting $\phi.$ By the pigeonhole principle, there exists a fixed point $x \in \fix(\phi),$ and an increasing sequence $k_i$ such that $c(\mu, H^{(r_i)})$ for $r_i =  2^{k_i}$ is carried by a capping $y_i$ of the $1$-periodic orbit of the $r_i$-iterated Hamiltonian $H^{(r_i)}$ corresponding to $x^{(r_i)}.$ By taking a power of $\phi,$ we can assume that $r_1 = 1,$ and set $y= y_1.$ Write $y_i$ as a recapped iteration of $y,$ that is \begin{equation}\label{eq: recap} y_i = y^{(r_i)} \# A_i.\end{equation} We claim that for $r_i$ large, $\om(A_i) \leq 0,$ and $c_1(A_i) > 0$ contradicting  monotonicity. Indeed, write $\cA_i$ for the action functional of $H^{(r_i)},$ and $\cA:=\cA_1.$ Then by \eqref{eq: recap} and Theorem \ref{thm: spec ineq}, \[r_i \cA(y) - \om(A_i) = \cA_i(y_i) = c(\mu, H^{(r_i)}) \geq r_i c(\mu, H) = r_i \cA(y).\]Hence \[\om(A_i) \leq 0.\] However, we know that $y_i$ carries $c(\mu, H^{(r_i)}),$ hence $\Delta(H^{(r_i)},y_i) \in (0,2n)$ and also $\Delta(H,y) \in (0,2n).$ Hence $r_i \Delta(H,y) > 2n$ for $r_i$ large enough, and \[2n > \Delta(H^{(r_i)}, y_i) = r_i \Delta(H, y) - 2c_1(A_i).\] Therefore \[c_1(A_i) > 0.\] \end{proof}

\section*{Acknowledgements}
I thank Paul Seidel for a very useful discussion, and Erman \c{C}ineli, Viktor Ginzburg, and Ba\c{s}ak G\"{u}rel for sharing their preprint \cite{CGG2} with me. At the University of Montr\'{e}al I was supported by an NSERC Discovery Grant, by the Fonds de recherche du Qu\'{e}bec - Nature et technologies, and by the Fondation Courtois.

\bibliographystyle{abbrv}
\bibliography{bibliographyCM}

\end{document}